\documentclass{article}

\usepackage[margin=1.3in]{geometry}

\usepackage{amsmath,amssymb,amsthm,mathrsfs,esint,color,times,textcomp,sectsty,verbatim,graphicx,enumerate,authblk,amsxtra,bm}
\usepackage{xcolor}
\usepackage[colorlinks=true]{hyperref}
\hypersetup{urlcolor=blue, citecolor=red, linkcolor=blue}
\usepackage[toc,page]{appendix}
\numberwithin{equation}{section}
\theoremstyle{plain}
\newtheorem{theorem}{Theorem}[section]

\newtheorem{lemma}[theorem]{Lemma}

\newtheorem{corollary}[theorem]{Corollary}

\theoremstyle{definition}

\newtheorem{remark}[theorem]{Remark}

\newcommand{\mr}{\mathbb{R}}
\newcommand{\ud}{\mathrm{d}}
\allowdisplaybreaks

\begin{document}
		\title{\Large \bf A Liouville-type theorem in conformally invariant equations}
	\author{ Mingxiang Li\thanks{M. Li, Email: limx@smail.nju.edu.cn. } }
	\affil{Department of Mathematics, Nanjing University, Nanjing 210093, P. R. China}
	\date{}
	\maketitle
\begin{abstract}
	Given a smooth function $K(x)$ satisfying a polynomially cone condition and $x\cdot\nabla K\leq 0$, we prove that there is no solution $u\in C^\infty(\mr^2)$ of the equation
	$$-\Delta u=K(x)e^{2u}\quad \mathrm{on}\;\mr^2$$
	with $u\leq C$  and $\int_{\mr ^2}|K(x)|e^{2u}\ud x<+\infty$.  As a consequence, there is no such solution if $K(x)$ is a non-constant polynomial with $x\cdot\nabla K\leq 0$. The latter result already includes a result of Struwe(JEMS 2020) as a particular case. Higher order cases are set up with additional assumption on the behavior of $\Delta u$ near infinity. 
\end{abstract}	

\medskip

\section{Introduction}
Recently,  Borer, Galimberti and   Struwe \cite{BGS} studied the asymptotic behavior of the metrics of a family of sign-changing prescribed Gaussian curvature on $(M,g)$ with $\chi(M)<0$ and then Galimberti \cite{Ga} studied it with $\chi(M)=0$. Their blow-up analysis led possibly to either
solutions of the standard Liouville equation $-\Delta u=e^{2u}$ in $\mr^2$ or  solutions to
$$-\Delta u=(1+(Ax,x))e^{2u}\quad \mathrm{in}\;\mr^2$$
with $\int_{\mr^2}(1+|A(x,x)|)e^{2u}\ud x<\infty$ where $A$ is a negative definite $2\times 2$-matrix. Later, Struwe \cite{Struwe20} ruled out the possibility of the later case. In \cite{Struwe20}, Struwe established an  elegant result as follows.
\begin{theorem}\label{thm: Struwe's theorem}(Theorem 1.3 in \cite{Struwe20})
	Suppose $A$ is a negative definite and symmetric $2\times 2$-matrix. Then there is no solution $u\in C^{\infty}(\mr^2)$ of the equation
	$$-\Delta u=(1+(Ax,x))e^{2u}\quad \mathrm{on}\;\mr^2$$
with $u\leq C$ and such that the induced metric $h=e^{2u}g_{\mr^2}$ has finite volume and integrated curvature
$$\int_{\mr ^2}e^{2u}\ud x<\infty,\quad \int_{\mr ^2}(1+(Ax,x))e^{2u}\ud x\in \mr.$$	
\end{theorem}
Inspired by the above theorem, we establish our main theorem in current paper.
\begin{theorem}\label{thm:main theorem of order 2}
	Suppose $p(x)$ is a non-constant polynomial with $x\cdot\nabla p\leq 0$. Then there is no solution $u\in C^{\infty}(\mr^2)$ of the equation
	\begin{equation*}
		-\Delta u=p(x)e^{2u}\quad \mathrm{on}\;\mr^2
	\end{equation*}
 with $u\leq C$ and such that the induced metric $h=e^{2u}g_{\mr^2}$ has finite total curvature
$$\int_{\mr ^2}|p(x)|e^{2u}\ud x<\infty.$$
\end{theorem}

In conformal geometry, the following equation
\begin{equation}\label{equ:conformal equation}
	-\Delta u=K(x)e^{2u}\quad \mathrm{in}\;\mr^2
\end{equation}
plays a very important role. In particular, when $K(x)$ is a  positive constant, the classification theorem in \cite{CL91} gives a very nice description of the solutions with finite volume. When $K(x)\leq 0$ and $K(x)\geq -C/|x|^l$ with $l>2$, Ni \cite{Ni82} showed that the equation \eqref{equ:conformal equation} possesses infinitely many solutions. Later, Cheng and Ni \cite{Cheng-Ni} showed that the equation \eqref{equ:conformal equation} has no solution when $K(x)\leq 0$ and $K(x)\leq -C|x|^{-2}$ near infinity. Not long after that, Chen and Li \cite{CL93} considered $K(x)$ positive near infinity and gave a nice description of the asymptotic behavior of the solutions to \eqref{equ:conformal equation} as well as a generalized Pohozev's  identity similar to Lemma \ref{lem:Pohozaev indentity}.  The property that $K(x)$ is positive near infinity is crucial in their proof. In \cite{CLin}, Cheng and Lin gave a more detailed discussion about the asymptotic behavior of the solutions 
including both $K(x)$ is positive and negative at infinity with finite total curvature. McOwen \cite{Mc} and Aviles \cite{Av} investigated the situation where $K(x)\to 0$ in some order as $|x|\to\infty$. If we restrict the equation on bounded domian, there are also many well-known results including \cite{B-M}, \cite{LS}, \cite{L} and many others.  Via stereographic projection, we can pull back the metric $e^{2u}g_{\mr^2}$ to the sphere $\mathbb{S}^2$ and then it becomes a well-known Nirenberg problem. Many distinguished works are devoted to it such as \cite{CY}, \cite{XY}, \cite{Struwe05} and many others. 

As for higher order cases, Galimberti \cite{Ga17} and Ngô-Zhang \cite{NZ} studied the related problem about Q-curvature in dimension four. Recently, Struwe \cite{Struwe21} established a fourth order  theorem like Theorem \ref{thm: Struwe's theorem}. In current paper, we will generalize it to all higher  order cases. Consider the higher order equation
$$(-\Delta)^{n/2} u=K(x)e^{nu}\quad \mathrm{on}\;\mr^n$$
with even $n\geq 4$. Similar to the two-dimensional case, there are also famous classification theorems when $K$ is a positive constant, see \cite{Lin}, \cite{WX} and so on.

Actually, we will prove a more general theorem compared with Theorem \ref{thm:main theorem of order 2}. Now, we define a polynomially cone condition(denoted as $\mathbf{(PCC)}$ for brevity) as follows:
For a given continuous  function $K(x)$ with $x\in \mr^n$, there exist $R_1>0$, $\theta>0$, $\gamma\geq n$ and $c_1>0,c_2>0$ such that  $|K(x)|\leq c_2|x|^\gamma$ for $|x|\geq R_1$ and
$$|K(x)|\geq c_1|x|^\gamma\quad \mathrm{in}\; C_\theta\backslash B_{R_1}(0)$$
where $C_\theta$ is a cone with angle $\theta$ and vertex at the origin.

When  $p(x)$ is a non-constant polynomial in $\mr^2$ with $x\cdot\nabla p\leq 0$, it is not hard to see $\deg (p(x))\geq 2$ and then $p(x)$ satisfies $\mathbf{(PCC)}$ with help of Lemma \ref{lem:polynomial satisfy PCC}. Therefore, Theorem \ref{thm:main theorem of order 2} is a corollary of Theorem \ref{thm:general theorem}. Now, we state the general theorem as follows.
\begin{theorem}\label{thm:general theorem}
	Suppose $K(x)$ is a smooth function on $\mr^2$ satisfying $\mathbf{(PCC)}$ and $x\cdot\nabla K\leq 0$. Then
	there is no solution $u\in C^{\infty}(\mr^2)$ of the equation
	\begin{equation}\label{equ:general}
		-\Delta u=K(x)e^{2u}\quad \mathrm{on}\;\mr^2
	\end{equation}
	with $u\leq C$ and
	 such that the induced metric $h=e^{2u}g_{\mr^2}$ has  finite total curvature
	$$\int_{\mr^2}|K(x)|e^{2u}\ud x<\infty.$$
\end{theorem}

For higher order cases, we can also get similar result with additional assumption on the behavior of $\Delta u$ near infinity like that in \cite{Struwe21}. Precisely, we state it as follows.
\begin{theorem}\label{thm:higher order}
		Suppose $K(x)$ is a smooth function on $\mr^n$ satisfying $\mathbf{(PCC)}$ and $x\cdot\nabla K\leq 0$  where  $n\geq 4$ is an even integer. Then
	there is no solution $u\in C^{\infty}(\mr^n)$ of the equation
	\begin{equation}\label{equ:general higher}
		(-\Delta)^{n/2} u=K(x)e^{nu}\quad \mathrm{on}\;\mr^n
	\end{equation}
	with $u\leq C$ and \begin{equation}\label{condtion laplican }
		\Delta u(x)\to 0 \quad\mathrm{as}\; |x|\to\infty
	\end{equation}
 such that the induced metric $h=e^{2u}g_{\mr^n}$ has  finite total curvature
	$$\int_{\mr^n}|K(x)|e^{nu}\ud x<\infty.$$
\end{theorem}
\begin{remark}
	In \cite{Struwe21}, Struwe considered  $K(x)=1+A(x,x,x,x)$ with $A$ being a negative definite and symmetric 4-linear map on $\mr^4$. Based on the definition of $A$ in Struwe's work,  for any $x\in \mr^4$, we have
	$$c_2|x|^4\geq |A(x,x,x,x)|\geq c_1|x|^4$$
	for some $c_1>0, c_2>0$. Using this property, one has
	$$|1+A(x,x,x,x)|\geq c_1|x|^4-1\geq \frac{c_1}{2}|x|^4$$
	for some $R_1>0$ and any  $|x|\geq R_1>0$ which shows that   $1+A(x,x,x,x)$ satisfies $\mathbf{(PCC)}$.  Therefore Theorem 1 in \cite{Struwe21} is a special case of Theorem \ref{thm:higher order}.
\end{remark}
Since this paper is short, we need not mention the structure of the paper as it goes naturally.
\section{Crucial Lemmas}
Firstly, we consider the integral equation
\begin{equation}\label{equ:integral equation}
	w(x)=\frac{2}{\omega_n}\int_{\mr^n}\left(\log\frac{|y|}{|x-y|}\right)K(y)e^{nw(y)}\ud y
\end{equation}
with $K(x)e^{nw(x)}\in L^1(\mr^n)$  where $\omega_n$ is the volume of unit sphere in $\mr^{n+1}$ and we set 
$$\alpha=\frac{2}{\omega_n}\int_{\mr^n}K(y)e^{nw(y)}\ud y$$
for brevity.

 Inspired by the proof of Theorem  4.5 in \cite{Struwe20}, we give the estimate of $\alpha$ as follows.
 \begin{lemma}\label{lem:alpha}
 		Given a function $K(x)\in C^1(\mr^n)$ where $n\geq 2$ is an even integer satisfying $\mathbf{(PCC)}$ and $x\cdot\nabla K\leq 0$.  Suppose that  $w(x)\in L^\infty_{loc}(\mr^n)$ bounded from above is a  solution of \eqref{equ:integral equation}.  Then we have
 	$$\alpha\geq 1+\frac{\gamma}{n}\geq 2$$
 	and 
 	$$\lim_{|x|\to\infty}\frac{w(x)}{\log|x|}=-\alpha.$$
 	Moreover,
 	$$w(x)\leq -\alpha\log |x|+C$$
 	for $|x|>>1$.
 \end{lemma}
 \begin{proof}
 	Choose $|x|\geq e^4$ such  that $|x|\geq 2\log|x|$. We split $\mr^n$ into three pieces
 	$$A_1=B_{1}(x), \quad A_2=B_{\log|x|}(0),\quad A_3=\mr^n\backslash (A_1\cup A_2).$$
 	For $y\in A_2$ and $|y|\geq2$ , we have $|\log\frac{|x|\cdot|y|}{|x-y|}|\leq \log(2\log|x|)$.
 	Respectively,  for $|y|\leq 2$, $ |\log\frac{|x|\cdot|y|}{|x-y|}|\leq |\log|y||+C$. 	Here and thereafter, we denote by $C$ a constant which may be different from line to line.
 	Thus 
 	\begin{equation}\label{A_2}
 		|\int_{A_2}\log\frac{|y|}{|x-y|}Ke^{nw}\ud y+\log|x|\int_{A_2}Ke^{nw}\ud y|\leq C\log\log|x|+C=o(1)\log|x|.
 	\end{equation}
 	as $|x|\to \infty$.
 	For $y\in A_3$, it is not hard to check 
 	$$\frac{1}{|x|+1}\leq\frac{|y|}{|x-y|}\leq |x|+1.$$ With help  of this estimate, we control the integral over $A_3$ as
 	\begin{equation}\label{A_3}
 		|\int_{A_3}\log\frac{|y|}{|x-y|}Ke^{nw}\ud y|\leq \log(|x|+1)\int_{A_3}|K|e^{nw(y)}\ud y.
 	\end{equation}
 	For $y\in B_1(x)$, one has
 	$1\leq |y|\leq |x|+1$ and then
 	$$|\int_{A_1}\log|y|Ke^{nw}\ud y|\leq \log(|x|+1)\int_{A_1}|K|e^{nw}\ud y.$$
 	It is not hard to see  $\int_{A_3\cup A_1}|K|e^{nw(y)}\ud y\to 0$ as $|x|\to \infty$.  Thus we have
 	\begin{equation}\label{u =-aplha log x}
 	\frac{2}{\omega_n}	\int_{\mr^n}\log\frac{|y|}{|x-y|}K(y)e^{nw(y)}\ud y=(-\alpha+o(1))\log|x|+\frac{2}{\omega_n}\int_{A_1}\log\frac{1}{|x-y|}K(y)e^{nw(y)}\ud y.
 	\end{equation}
 	Due to $Ke^{nw}\in L^1(\mr^n)$ and Fubini's theorem, one has  
 	$$\int_{B_1(x_0)}|\int_{A_1}\log\frac{1}{|x-y|}K(y)e^{nw(y)}\ud y|\ud x\leq C.$$
 	Thus
 	\begin{equation}\label{|u|}
 		\int_{B_1(x_0)}|\int_{\mr^n}\log\frac{|y|}{|x-y|}Ke^{nw}\ud y|\ud x\leq C\log(|x_0|+2)
 	\end{equation}
 	which will be used in our proofs of main theorems later.

 	Since $K$ satisfies $\mathbf{(PCC)}$, we could find $0<r_\theta<1/2$ fixed such that $B_\theta:=B_{r_\theta|x_0|}(x_0)$ lying in the cone where$|K(x)|\geq c_1|x|^\gamma$  as $|x_0|\to \infty$ suitably. 
 	With help of Fubini's theorem, direct computation yields that
 	\begin{align*}
 		&|\int_{B_\theta}\int_{|x-y|\leq 1}\log\frac{1}{|x-y|}Ke^{nw}\ud y\ud x|\\
 		\leq &\int_{B_\theta}\int_{|x-y|\leq 1}\frac{1}{|x-y|}|K|e^{nw}\ud y\ud x\\
 		\leq &\int_{B_{r_\theta|x_0|+1}(x_0)}|K|e^{nw}\int_{B_\theta}\frac{1}{|x-y|}\ud x\ud y\\
 		\leq &C|x_0|^{n-1}.
 	\end{align*}
 	Meanwhile, notice that 
 	$1-r_\theta\leq\frac{|x|}{|x_0|}\leq 1+r_\theta$ for $x\in B_{\theta}$.
 	Thus
 	$$\frac{1}{|B_\theta|}\int_{B_\theta}w(x)\ud x=(-\alpha +o(1))\log|x_0|$$
 	where we use the fact $w(x)=\frac{2}{\omega_n}\int_{\mr^n}\log\frac{|y|}{|x-y|}Ke^{nw}\ud y$.
 	With help of Jensen's inequality, we have
 	$$e^{\frac{1}{|B_\theta|}\int_{B_\theta}nw(x)\ud x}\leq\frac{1}{|B_\theta|} \int_{B_\theta}e^{nw}\ud x.$$
 	Recall $|K|\geq c_1 |x|^\gamma$ in $B_\theta$ and then one has
 	$$|x_0|^{n+\gamma-n\alpha +o(1)}\leq C \int_{B_\theta}|K|e^{nu}\ud x.$$
 	Notice that the right side of above inequality tends to zero as $|x_0|\to\infty$ based on  $Ke^{nu}\in L^1(\mr^n)$.
 	Thus we have
 	\begin{equation}\label{alpha geq 1+gamma/n}
 		\alpha\geq 1+\frac{\gamma}{n} \geq 2.
 	\end{equation}
 	Since $x\cdot\nabla K\leq 0$, it is easy to get $K\leq C$.
 	Combing with $w\leq C$, we  get the bound
 	\begin{equation}\label{K^+upper bound}
 		0\leq \int_{A_1}\log\frac{1}{|x-y|}K^+e^{nw}\ud y\leq C.
 	\end{equation}
 	Then we rewrite \eqref{u =-aplha log x} as
 	\begin{equation}\label{u=-alpha log +u_2}
 		w(x)=(-\alpha +o(1))\log|x|-\frac{2}{\omega_n}\int_{A_1}\log\frac{1}{|x-y|}K^-e^{nw}\ud y.
 	\end{equation}
 	Immediately, $w(x)\leq (-\alpha+o(1))\log|x|$. Since $\alpha \geq 1+\frac{\gamma}{n}$ and $|K|\leq C|x|^{\gamma}$ for sufficiently large $|x|$, we have
 	$$|K|e^{nw}\leq C$$
 	and then
 	$$|\frac{2}{\omega_n}\int_{A_1}\log\frac{1}{|x-y|}K^-e^{nw}\ud y|\leq C.$$
 	for sufficiently large $|x|$.
 	Finally, we have
 	\begin{equation}\label{u=-alpha }
 		\lim_{|x|\to\infty}\frac{w(x)}{\log|x|}=-\alpha.
 	\end{equation}
 	
 	Now we are going to give a more precise upper bound of $w$ as following
 	$$w(x)\leq -\alpha\log |x|+C.$$
 	Firstly, direct computation yields that
 	\begin{align*}
 		&w(x)+\alpha \log |x|\\
 		=&\frac{2}{\omega_n}\int_{\mr^n}\log\frac{|x|\cdot(|y|+1)}{|x-y|}K^+e^{nw}\ud y\\
 		&+\frac{2}{\omega_n}\int_{\mr^n}\log\frac{|x|\cdot(|y|+1)}{|x-y|}(-K^-)e^{nw}\ud y\\
 		&+\frac{2}{\omega_n}\int_{\mr^n}\log\frac{|y|}{|y|+1}Ke^{nw}\ud y\\
 		=:&I_1+I_2+I_3
 	\end{align*}	
 	
 	For $|x|\geq 1$, it is easy to check 
 	$$\frac{|x|\cdot(|y|+1)}{|x-y|}\geq 1$$ and then
 	$$I_2\leq 0.$$
 	Based on our assumptions, one has $I_3=C_0$ for some constant. 
 	
 	As for the term $I_1$, we need  use again the fact $K^+\leq C$. By using 	\eqref{u=-alpha } and \eqref{alpha geq 1+gamma/n}, 
 	we have
 	\begin{equation}\label{e^nu leq |y|^-3n/2}
 		e^{nw(y)}\leq C|y|^{-\frac{3n}{2}}
 	\end{equation}
 	for $|y|\geq1$.
 	We split $I_1$ as 
 	\begin{align*}
 		I_1=&\int_{|x-y|\leq \frac{|x|}{2}}\log\frac{|x|\cdot(|y|+1)}{|x-y|}K^+e^{nw}\ud y\\
 		&+\int_{|x-y|\geq \frac{|x|}{2}}\log\frac{|x|\cdot(|y|+1)}{|x-y|}K^+e^{nw}\ud y\\
 		=&:II_1+II_2.
 	\end{align*}
 	We  deal with these terms one by one.
 	Notice  that $|x|/2\leq|y|\leq 3|x|/2$ when $|x-y|\leq |x|/2$. Then
 	\begin{align*}
 		II_1=&\int_{|x-y|\leq \frac{|x|}{2}}\log|x|\cdot(|y|+1)K^+e^{nw}\ud y+\int_{|x-y|\leq \frac{|x|}{2}}\log\frac{1}{|x-y|}K^+e^{nw}\ud y\\
 		\leq& C \log (|x|(\frac{3|x|}{2}+1))|x|^{-3n/2}\int_{|x-y|\leq |x|/2}\ud y+C|x|^{-3n/2}\int_{|x-y|\leq |x|/2}\frac{1}{|x-y|}\ud y\\
 		\leq &C\log(3|x|^2/2+|x|)|x|^{-n/2}+C|x|^{-n/2+1}\leq C
 	\end{align*}
 	for $|x|>>1$. As for the term $II_2$, we make use  of $|x|/|x-y|\leq 2$ and  \eqref{e^nu leq |y|^-3n/2} to get
 	$$II_2\leq \int_{|x-y|\geq \frac{|x|}{2}}\log(2|y|+2)K^+e^{nw}\ud y\leq C.$$
 	Combining above estimates, we finally get the desired results
 	$$w(x)+\alpha\log |x|\leq C$$
 	for $|x|>>1.$
 \end{proof}

The following lemma plays an important role in the proof of our main theorems. The following generalized Pohozaev identity is also got in \cite{CLin} and \cite{CL93} with additional assumptions but the proof in \cite{Xu05} is very different from others. Inspired by \cite{Xu05}, we give the following Pohozaev indentity. Moreover, with help of the ideas from \cite{Li04} and \cite{CLO} related to intergral equations and elliptic theory, we generalize Theorem 1.1 in \cite{Xu05} with lower regularity assumption of the solution.

	\begin{lemma}\label{lem:Pohozaev indentity}
	Given a function $K(x)\in C^1(\mr^n)$ where $n\geq 2$ is an even integer satisfying $\mathbf{(PCC)}$ and $x\cdot\nabla K\leq 0$.  Suppose $w(x)\in L^\infty_{loc}(\mr^n)$ bounded from above is a  solution of \eqref{equ:integral equation}.
	Then
	\begin{equation}\label{equ:Pohozaev identity}
		\frac{4}{n\omega_n}\int_{\mr^n}\langle x,\nabla K(x)\rangle e^{nw(x)}\ud x=\alpha(\alpha-2).
	\end{equation}
\end{lemma}

\begin{proof}
		Firstly, we will show that $w\in C^1(\mr^n)$. Here, we just assume that  $w\in L^\infty_{loc}(\mr^n)$. Actually,   with help of   Lemma \ref{appen.1}, we can deduce that $w\in C^1(\mr^n)$ when $Ke^{nw}\in L^1(\mr^n)$, $K(x)\in C^1(\mr^n)$ and $w\in L^\infty_{loc}(\mr^n)$. 
	By direct computation, one has
	\begin{equation}\label{equ:x,nabla w}
		\langle x,\nabla w\rangle=-\frac{2}{\omega_n}\int_{\mr^n}\frac{\langle x,x-y\rangle}{|x-y|^2}K(y)e^{nw(y)}\ud y
	\end{equation}
	Multiplying by $K(x)e^{nw(x)}$ and integrating over the ball $B_R(0)$ for any $R>0$, we have
	\begin{equation}\label{equ:ingegrate x,nabla w}
		\int_{B_R(0)}K(x)e^{nw(x)}\left[-\frac{2}{\omega_n}\int_{\mr^n}\frac{\langle x,x-y\rangle}{|x-y|^2}K(y)e^{nw(y)}\ud y\right]\ud x=\int_{B_R(0)}K(x)e^{nw(x)}\langle x,\nabla w(x)\rangle\ud x.
	\end{equation}
	With $x=\frac{1}{2}\left((x+y)+(x-y)\right)$, for the left-hand side of \eqref{equ:ingegrate x,nabla w}, one has the following identity
	\begin{align*}
		LHS=&\frac{1}{2}\int_{B_R(0)}K(x)e^{nw(x)}\left[-\frac{2}{\omega_n}\int_{\mr^n}K(y)e^{nw(y)}\ud y\right]\ud x   \\
		&+\frac{1}{2}\int_{B_R(0)}K(x)e^{nw(x)}\left[-\frac{2}{\omega_n}\int_{\mr^n}\frac{\langle x+y,x-y\rangle}{|x-y|^2}K(y)e^{nw(y)}\ud y\right]\ud x.
	\end{align*}
Now, we deal with  the last term of above equation by changing variables $x$ and $y$. 
	\begin{align*}
		&	\int_{B_R(0)}K(x)e^{nw(x)}\left[-\frac{2}{\omega_n}\int_{\mr^n}\frac{\langle x+y,x-y\rangle}{|x-y|^2}K(y)e^{nw(y)}\ud y\right]\ud x\\
		=&\int_{B_R(0)}K(x)e^{nw(x)}\left[-\frac{2}{\omega_n}\int_{\mr^n\backslash B_R(0)}\frac{\langle x+y,x-y\rangle}{|x-y|^2}K(y)e^{nw(y)}\ud y\right]\ud x\\
		=&\int_{ B_{R/2}(0)}K(x)e^{nw(x)}\left[-\frac{2}{\omega_n}\int_{\mr^n\backslash B_R(0)}\frac{\langle x+y,x-y\rangle}{|x-y|^2}K(y)e^{nw(y)}\ud y\right]\ud x\\
		&+\int_{B_R(0)\backslash B_{R/2}(0)}K(x)e^{nw(x)}\left[-\frac{2}{\omega_n}\int_{\mr^n\backslash B_{2R}(0)}\frac{\langle x+y,x-y\rangle}{|x-y|^2}K(y)e^{nw(y)}\ud y\right]\ud x\\
		&+\int_{B_R(0)\backslash B_{R/2}(0)}K(x)e^{nw(x)}\left[-\frac{2}{\omega_n}\int_{B_{2R(0)}\backslash B_{R}(0)}\frac{\langle x+y,x-y\rangle}{|x-y|^2}K(y)e^{nw(y)}\ud y\right]\ud x\\
		=:&I_1+I_2+I_3.
	\end{align*} Notice that 
	$$|I_1|\leq 3\frac{2}{\omega_n}\int_{ B_{R/2}(0)}|K(x)|e^{nw(x)}\ud x\int_{\mr^n\backslash B_R(0)}|K(y)|e^{nw(y)}\ud y $$
	and
	$$|I_2|\leq 3\frac{2}{\omega_n}\int_{ B_R(0)\backslash B_{R/2}(0)}|K(x)|e^{nw(x)}\ud x\int_{\mr^n\backslash B_{2R}(0)}|K(y)|e^{nw(y)}\ud y.$$
	Then both $|I_1|$ and $|I_2|$ tend to zero as $R\to \infty$ due to $Ke^{nw}\in L^1(\mr^n)$.
	
	Now,  we focus on the term  $I_3$. With help of Lemma \ref{lem:alpha}, for each $x\in B_R(0)\backslash B_{R/2}(0)$ and $K$ satisfying $\mathbf{(PCC)}$, one has $|K|e^{nw(x)}\leq C|x|^{-n}$ and then
	\begin{align*}
		&\int_{B_{2R(0)}\backslash B_{R}(0)}\frac{| x+y|}{|x-y|}|K(y)|e^{nw(y)}\ud y\\
		\leq &CR^{1-n}\int_{B_{2R(0)}\backslash B_{R}(0)}\frac{1}{|x-y|}\ud y\\
		\leq &CR^{1-n}\int_{B_{3R}(0)}\frac{1}{|y|}\ud y\leq C.
	\end{align*} 
	Thus
	$$|I_3|\leq C\int_{ B_R(0)\backslash B_{R/2}(0)}|K|e^{nw}\ud x\to 0 $$
	as $R\to \infty.$
	
	Then we have
	$$LHS\to -\frac{1}{2}\alpha\int_{\mr^n}K(y)e^{nw(y)}\ud y.$$
	As for the right-hand side of \eqref{equ:ingegrate x,nabla w}, by using divergence theorem, we have
	\begin{align*}
		RHS=&\frac{1}{n}\int_{B_R(0)}K(x)\langle x,\nabla e^{nw(x)}\rangle\ud x   \\
		=&-\int_{B_R(0)}\left( K(x)+\frac{1}{n}\langle x,\nabla K(x)\rangle\right)e^{nw(x)}\ud x\\
		&+\frac{1}{n}\int_{\partial B_R(0)}K(x)e^{nw(x)}R\ud \sigma.
	\end{align*}
	Since $K(x)e^{nw(x)}\in L^1(\mr^n)$, the last term of above equation tends to zero as $R\to\infty$ suitably.  Precisely, there exist a sequence $R_i\to \infty$ such that 
	$$\lim_{i\to\infty}R_i\int_{\partial B_{R_i}(0)}Ke^{nw}\ud\sigma=0.$$
	Then 
	$$\lim_{i\to\infty}\frac{4}{n\omega_n}\int_{B_{R_i}(0)}x\cdot \nabla K e^{nw}\ud x= \alpha(\alpha-2).$$
	Due to $x\cdot \nabla K\leq 0$,  we have
	$$\frac{4}{n\omega_n}\int_{\mr^n}x\cdot\nabla K e^{nw}\ud x=\alpha(\alpha-2).$$
	
\end{proof}

Combining Lemma \ref{lem:Pohozaev indentity} and Lemma \ref{lem:alpha}, it seems that we have got our desired result based on the assumption $x\cdot\nabla K\leq 0$. Actually, if we just consider the solution of the integral equation, we have gotten the non-existence theorem which generalizes a result of Hyder and Martinazzi \cite{HM}.
\begin{theorem}(Proposition 2.8  in \cite{HM})\label{thm:HM}
	For $K(x)=1-|x|^p,\;x\in \mathbb{R}^4$, $p\geq 4$, there is no solution to the integral eqaution
	$$u(x)=\frac{1}{8\pi^2}\int_{\mr^4}\log(\frac{|y|}{|x-y|})K(y)e^{4u(y)}\ud y +c$$
	for some $c\in \mathbb{R}$ with $Ke^{4u}\in L^1(\mathbb{R}^4)$.
\end{theorem}

Actually, in this sepcial case $K(x)=1-|x|^p$  with $p\geq 4$, it is easy to check that  $x\cdot\nabla K\leq 0$  and $K(x)$ satisfies $\mathbf{(PCC)}$. As for the harmless constant $c$, we can just do a scaling over $K$ like the strategy taken in the proof of Theorem \ref{thm:general theorem} such that $c=0$. It is obvious that the conditions $x\cdot\nabla K\leq 0$ and $\mathbf{(PCC)}$ still hold after a scaling.  With help of Lemma \ref{lem:Pohozaev indentity} and Lemma \ref{lem:alpha}, we have the following corollary.

\begin{corollary}\label{corollary of H-M}
	Given a function $K(x)\in C^1(\mr^n)$ satisfying $\mathbf{(PCC)}$ and $x\cdot\nabla K(x)\leq 0$. There is no   solution to the integral equation \eqref{equ:integral equation} with $Ke^{nu}\in L^1(\mathbb{R}^n)$, $u\in L^\infty_{loc}(\mr^n)$ and  $u\leq C$.
\end{corollary}
\begin{remark}
	We should point out that we assume $u\leq C$ in addition compared with Theorem \ref{thm:HM}. Actually, if $K^+$ has compact support, we can remove this assumption by noticing that the estimate \eqref{K^+upper bound} is trivial in this case. With help of $u\leq C$, we could deal with the degenerate cases such as $K(x)=1-x_1^n$.
\end{remark}

 However, as for the conformal equation itself, we still need to study the connection between the integral equation and the conformal equation.
Before we finish our proof of main theorems, we establish the following two useful lemmas for readers' convenience.  
\begin{lemma}\label{lem:polynomial satisfy PCC}
	Suppose $p(x)$ is a polynomial on $\mr^n$ and $deg(p(x))\geq n$, then $p(x)$ satisfies $\mathbf{(PCC)}$.
\end{lemma}
\begin{proof}
	Denote $k:=deg(p(x))\geq n$ and separate $p(x)$ as 
	$$p(x)=p_k(x)+l(x)$$
	where $p_k(x)$ is a homogeneous polynomial of degree $k$ and $\deg (l(x))\leq k-1$.
	We choose  the polar coordinate $(r,\theta)$ with $r\geq 0$ and $\theta\in\mathbb{S}^{n-1}$. Then one has
	$$p_k(x)=r^{k}\varphi_k(\theta)$$ where $\varphi_k(\theta)$ is a non-zero smooth function defined on $\mathbb{S}^{n-1}$. There exist $c_1>0$ and a geodesic ball $B_{r_1}(\theta_0)\subset\mathbb{S}^{n-1}$ such that
	$$|\varphi_k(\theta)|\geq 2c_1>0$$
	for any $\theta\in B_{r_1}(\theta_0)$.  Then 
	$$|p(x)|\geq2c_1|x|^k-c_2|x|^{k-1}$$
	where $c_2>0$ is a constant depending only on the cofficients of $l(x)$. We choose $R_1=\max\{1,\frac{c_2}{c_1}\}$ and then one has
	$$|p(x)|\geq c_1|x|^k+|x|^{k-1}(c_1|x|-c_2)\geq c_1|x|^k\geq c_1|x|^n$$
	for any $x\in [R_1,+\infty)\times B_{r_1}(\theta_0)$. 
\end{proof}

For readers' convenience, we slightly improve   Theorem 5 in \cite{Ma} to be used later which is a generalization of Theorem 2.3 in \cite{ARS} with help of Pizzetti’s formula \cite{Pi}.
\begin{lemma}\label{lem:Pizzetti's formula}
	(Theorem 5 in \cite{Ma})
	Consider $h:\mr ^n\rightarrow \mr$ with $\Delta^m h=0$ and $h(x)\leq C(1+|x|^l)$, for some  $l\geq 0$. Then $h(x)$ is a polynomial of degree at most $\max\{l,2m-2\}$.
\end{lemma}
 \begin{proof}
 	If  $l\geq 2m-2$, there is nothing to do because this is exactly the result in \cite{Ma} which concludes that $h(x)$ is a polynomial of degree at most $l$. If $l<2m-2$, then we simply bound $h$ from above by a multiple of $1+|x|^{2m-2}$ which returns to the former case. Thus $h(x)$ is a polynomial of degree at most $\max\{l,2m-2\}$.
 \end{proof}

\section{Proof of Theorems}

{\bf Proof of Theorem \ref{thm:general theorem}:}
\begin{proof}
	Following Chen-Li \cite{CL91}, we introduce
	$$\tilde u(x)=\frac{1}{2\pi}\int_{\mr^2}\log\frac{|y|}{|x-y|}K(y)e^{2u}\ud y.$$
	Then the function $v:=u-\tilde u$ is harmonic.
 By using mean value  property of harmonic function, one has
 $$v(x_0)=\frac{1}{|B_1(x_0)|}\int_{B_1(x_0)}(u(x)-\tilde u(x))\ud x\leq C+C\log(2+|x_0|)$$
 where we have used the estimate \eqref{|u|} and the assumption $u\leq C$.
   Therefore, $v$ must be a constant.  This is a very classical result, which can be also obtained from Lemma \ref{lem:Pizzetti's formula} by choosing
$m=1$ and $l=1$.  It is not hard to check $v(x)\leq C(1+|x|)$. Then $v(x)$ is a polynomial of degree at most one. However,  the estimate $v(x)\leq C+C\log(2+|x|)$ rules out the case degree of one.
		Hence $\tilde u=u+C_1$ for $C_1\in\mr$. Then
		$$\tilde u=\frac{1}{2\pi}\int_{\mr^2}\log\frac{|y|}{|x-y|}K(y)e^{-2C_1}e^{2\tilde u}\ud y$$ satisfying the integral equation \eqref{equ:integral equation}.
		Then $K(y)e^{-2C_1}e^{2\tilde u}\in L^1(\mr^2)$ and $K(y)e^{-2C_1}$ satisfies $\mathbf{(PCC)}$ which shows that
		$$\frac{1}{2\pi}\int_{\mr^2}\langle x,\nabla K(x)\rangle e^{2u(x)}\ud x=\frac{1}{2\pi}\int_{\mr^2}\langle x,\nabla( K(x)e^{-2C_1})\rangle e^{2\tilde u(x)}\ud x=\alpha(\alpha-2)$$
		where $\alpha=\frac{1}{2\pi}\int_{\mr^2}K(x)e^{-2C_1}e^{2\tilde u}\ud x=\frac{1}{2\pi}\int_{\mr^2}K(x)e^{2 u}\ud x$. Since $x\cdot\nabla K\leq 0$ and $K(x)$ satisfying $\mathbf{(PCC)}$, we have
		$$\frac{1}{2\pi}\int_{\mr^2}\langle x,\nabla K(x)\rangle e^{2u(x)}\ud x<0$$
		which concludes that $0<\alpha<2$ contradicting to Lemma \ref{lem:alpha}.
\end{proof}
Now we imitate the proof of Theorem \ref{thm:general theorem}  to give the proof of higher order case.

{\bf Proof of Theorem \ref{thm:higher order}:}
\begin{proof}
	Following the argument in \cite{Struwe21}, we 
	consider 
	$$\tilde u=\frac{\alpha_n}{\omega_n}\int_{\mr^n}\left(\log\frac{|y|}{|x-y|}\right)K(y)e^{nu(y)}\ud y$$
	where $\alpha_n$ is a positive constant such that
	$$(-\Delta)^{n/2}\tilde u=K(x)e^{nu(x)}.$$
	Actually, $\tilde u$ is also smooth based on the smoothness  assumption of $K(x)$ and $u(x)$ as well as $Ke^{nu}\in L^1(\mr^n)$.
	Following the argument of the proof of Theorem 1.2 in \cite{Li04}, we consider $B_4(x_0)$ for any $x_0\in\mr^n$. Choose a smooth cut-off function $0\leq \eta(x)\leq 1$ such that $\eta(x)=1$ in $B_2(x_0)$ and  vanishes outside $B_4(x_0)$. For $x\in B_1(x_0)$, we can split $\tilde u$ by
	\begin{align*}
		\tilde u(x)&=\frac{\alpha_n}{\omega_n}\int_{\mr^n}\left(\log\frac{|y|}{|x-y|}\right)\eta(y)K(y)e^{nu(y)}\ud y+\frac{\alpha_n}{\omega_n}\int_{\mr^n}\left(\log\frac{|y|}{|x-y|}\right)(1-\eta(y))K(y)e^{nu(y)}\ud y\\
		&=\frac{\alpha_n}{\omega_n}\int_{\mr^n}\left(\log\frac{|y|}{|x-y|}\right)\eta(y)K(y)e^{nu(y)}\ud y+\frac{\alpha_n}{\omega_n}\int_{\mr^n\backslash B_2(x_0)}\left(\log\frac{|y|}{|x-y|}\right)(1-\eta(y))K(y)e^{nu(y)}\ud y\\
		&=:\uppercase\expandafter{\romannumeral1}(x)+\uppercase\expandafter{\romannumeral2}(x)
	\end{align*}
Since $\eta(x)K(x)e^{nu}\in C_c^\infty(\mr^n)$, with help of Theorem 10.3 in \cite{LL}, we have $\uppercase\expandafter{\romannumeral1}(x)\in C^\infty(B_1(x_0))$. Since $(1-\eta)Ke^{nu}\in L^1(\mr^n)$, we can differentiate $\uppercase\expandafter{\romannumeral2}(x)$ under the integral sign for $x\in B_1(x_0)$, so $\uppercase\expandafter{\romannumeral2}(x)\in C^\infty(B_1(x_0))$. Thus $\tilde u\in C^\infty(\mr^n)$. 
	Set $v=u-\tilde u$ and then $(-\Delta)^{n/2}v=0$.

	 If we show that $v$ is a constant, combining Lemma \ref{lem:Pohozaev indentity} and \ref{lem:alpha} will show the contradiction as the proof of Theorem \ref{thm:general theorem}.
	 
	  Direct computation yields that
	\begin{align*}
		|\int_{B_1(x_0)}(-\Delta )\tilde u\ud x|=&\frac{\alpha_n}{\omega_n}|\int_{B_1(x_0)}\int_{\mr^n}(-\Delta)_x\left(\log\frac{1}{|x-y|}\right)K(y)e^{u(y)}\ud y\ud x|   \\
		\leq & C\int_{\mr^n}\int_{B_1(x_0)}\frac{|K(y)|e^{nu(y)}}{|x-y|^{2}}\ud x\ud y\\
		=&C\int_{B_{|x_0|/2}(0)}\int_{B_1(x_0)}|K(y)|\frac{e^{nu(y)}}{|x-y|^{2}}\ud x\ud y\\
			&+  C\int_{\mr^n\backslash (B_{|x_0|/2}(0)\cup B_2(x_0))}\int_{B_1(x_0)}|K(y)|\frac{e^{nu(y)}}{|x-y|^{2}}\ud x\ud y\\
			&+C\int_{B_2(x_0)}\int_{B_1(x_0)}|K(y)|\frac{e^{nu(y)}}{|x-y|^{2}}\ud x\ud y  \\
			\leq & C(\frac{|x_0|}{2}-1)^{-2}+C\int_{\mr^n\backslash B_{|x_0|/2}(0)}K(y)e^{nu(y)}\ud y\\
			&+C\int_{B_2(x_0)}\int_{B_4(y)}|K(y)|\frac{e^{nu(y)}}{|x-y|^{2}}\ud x\ud y\\
			\leq& C(\frac{|x_0|}{2}-1)^{-2}+C\int_{\mr^n\backslash B_{|x_0|/2}(0)}K(y)e^{nu(y)}\ud y\\
			&+C\int_{B_2(x_0)}|K(y)|e^{nu(y)}\ud y\to 0 \quad \mathrm{as}\;|x_0|\to \infty.
	\end{align*}
Combing with the assumption \eqref{condtion laplican }, we have
\begin{equation}\label{Delta v to 0}
	|\int_{B_1(x_0)}\Delta v\ud x|\to 0,\quad  \mathrm{as} \; |x_0|\to \infty.
\end{equation}
With help of \eqref{|u|}, one has 
\begin{equation}\label{equ: v upper bound in higher order}
	\int_{B_1(x_0)}v^+\ud x\leq\int_{B_1(x_0)}u^++|\tilde u|\ud x \leq C+C\log(2+|x_0|).
\end{equation}
If $x_0$ fixed and $R>>|x_0|$, we could choose suitable  $B_1(x_i)$ where $x_i\in B_{2R}(0)$ such that $B_R(x_0)\subset \cup^N_{i=1}B_1(x_i)$ and   $B_{1/5}(x_i)$ are disjoint as well as  $\cup^N_iB_{1/5}(x_i)\subset B_{3R}(0)$.Then
\begin{align*}
	\frac{1}{|B_R(x_0)|}\int_{B_R(x_0)}v^+(y)\ud y\leq&\frac{1}{|B_R(x_0)|}\sum^N_{i=1}\int_{B_1(x_i)}v^+\ud x\\
	\leq &C\frac{1}{|B_R(x_0)|}\sum^N_{i=1}\int_{B_1(x_i)}\log (2R+2)\ud x\\
	=&C\frac{5^n}{|B_R(x_0)|}\int_{\cup^N_iB_{1/5}(x_i)}\ud x\log(2R+2)\\
	\leq &C\log(2R+2)
\end{align*}
Following the argument of Theorem 5 in \cite{Ma} and using Propostion 4 in \cite{Ma}, we have for any $x\in \mr^n$ fixed and $R>>|x|$,
\begin{align*}
	|D^{n-1}v(x)|\leq& \frac{C}{R^{n-1}}\frac{1}{|B_R(x)|}\int_{B_R(x)}|v(y)|\ud y\\
	=&\frac{C}{R^{n-1}}\left(-\frac{1}{|B_R(x)|}\int_{B_R(x)}v(y)\ud y+2\frac{1}{|B_R(x)|}\int_{B_R(x)}v^+(y)\ud y\right)\\
	\leq &\frac{C}{R^{n-1}}\left(\sum^{n/2-1}_{i=0}c_iR^{2i}\Delta^ih(x)+\log (2R+2)\right)\\
	=&O(R^{-1})
\end{align*}
where we use Pizzetti's formula for polyhamonic function $\Delta^{n/2}v=0$ (See Lemma 3 in \cite{Ma}).
Then by taking  $R\to\infty$, one has $D^{n-1}v(x)=0$ which implies that  $v$ is a polynomial of degree at most $n-2$.

Then $\Delta v$ is a polynomial of degree at most $n-4$. Applying the estimate \eqref{Delta v to 0} to get that $\Delta v=0$. Then recall the estimate \eqref{equ: v upper bound in higher order} to get that $v$ must be a constant similar to the argument in the proof of Theorem \ref{thm:general theorem}.
The remian argument is the same as that of Theorem \ref{thm:general theorem}. Finally, we complete our proof. 
\end{proof}

\appendix
\section{Appendix}

\begin{lemma}\label{appen.1}
Suppose $n$ is an even integer and  $K(x)\in C^{k,\alpha}_{loc}(\mr^n),$ with $k\geq 0$ and $0<\alpha<1$, each  solution to the intergal equation 
	\begin{equation}\label{equ:integral equation2}
	u(x)=\frac{2}{\omega_n}\int_{\mr^n}\left(\log\frac{|y|}{|x-y|}\right)K(y)e^{nu(y)}\ud y
\end{equation}
with $Ke^{nu}\in L^1(\mr^n)$ and $u\in L^\infty_{loc}(\mr^n)$, then  $u$  belongs to $C^{n+k,\alpha}_{loc}(\mr^n)$. In particularly, if $K(x)$ is smooth, so is $u(x)$.
\end{lemma}
\begin{proof}
	Given any $B_4(x_0)\subset\mr^n$, choose $R=2|x_0|+8$ and split the right side of the integral equation as
	\begin{align*}
		&\int_{\mr^n}\left(\log\frac{|y|}{|x-y|}\right)K(y)e^{nu(y)}\ud y\\
		=&\int_{B_R(0)}\left(\log\frac{|y|}{|x-y|}\right)K(y)e^{nu(y)}\ud y+\int_{\mr^n\backslash B_R(0)}\left(\log\frac{|y|}{|x-y|}\right)K(y)e^{nu(y)}\ud y\\
		=&:\uppercase\expandafter{\romannumeral1}(x)+\uppercase\expandafter{\romannumeral2}(x)
	\end{align*}
inspired by the proof of Theorem 1.2 in \cite{Li04}.
	It is not hard to check $|\log\frac{|y|}{|x-y|}|\leq C$ for $x\in B_4(x_0)$ and $y\in\mr^n\backslash B_R(0)$. Thus, $\uppercase\expandafter{\romannumeral2}(x)$ is well-defined for  $x\in B_4(x_0)$ with help of $Ke^{nu}\in L^1(\mr^n)$. Since we can can differentiate $\uppercase\expandafter{\romannumeral2}(x)$ under the integral sign for $x\in B_4(x_0)$, so $\uppercase\expandafter{\romannumeral2}(x)\in C^\infty(B_4(x_0))$.
	For some constant $\gamma_n$, $\frac{1}{\gamma_n}\log\frac{1}{|x|}$ is a fundamental solution of the operator $(-\Delta)^{n/2}$ in $\mr^n$ satisfying 
	\begin{equation}\label{Green's function}
		(-\Delta)^{n/2}\left(\frac{1}{\gamma_n}\log\frac{1}{|x|}\right)=\delta_0,\quad \mathrm{in}\;\mr^n.
	\end{equation}
 Due to \eqref{Green's function}, it is easy to see
	\begin{equation}\label{II(x)}
		(-\Delta)^{n/2}\uppercase\expandafter{\romannumeral2}(x)=0, x\in B_4(x_0).
	\end{equation}
	Since  $K\in C^{k,\alpha}_{loc}(\mr^n)$ and $u\in L^\infty_{loc}(\mr^n)$, both $\log|y| K(y)e^{nu(y)}$ and $\log|x-y|K(y)e^{nu(y)}$ are integrable over $B_R(0)$ for any $x\in B_4(x_0)$.  Then   
	$$\uppercase\expandafter{\romannumeral1}(x)=\int_{\mr^n}\log\frac{1}{|x-y|}K(y)e^{nu(y)}\chi_{B_R(0)}(y)\ud y+c_1$$
	where $c_1$ is a constant.
	With help of Theorem 6.21 in \cite{LL} and the fundamental solution \eqref{Green's function}, we have
 $$(-\Delta)^{n/2}\uppercase\expandafter{\romannumeral1}(x)=\gamma_nKe^{nu},\; x\in B_4(x_0)$$ 
 in the sense of distribution.
 Combing with \eqref{II(x)}, one has 
 $$(-\Delta)^{n/2}u(x)=\frac{2\gamma_n}{\omega_n}Ke^{nu},\; x\in B_4(x_0)$$
 in sense of distribution.
 
 Now, following the argument of Corollary 8 in \cite{Ma}, we write $u|_{B_4(x_0)}=v_1+v_2$, where 
 \begin{align*}
 	\left\{\begin{array}{ll}
 		(-\Delta)^{n/2}v_1=\frac{2\gamma_n}{\omega_n}Ke^{nu} & \mathrm{in}\; B_4(x_0)\\
 		v_1=\Delta v_1=\cdots=\Delta^{n/2-1}v_1=0 & \mathrm{on} \;\partial B_4(x_0)
 	\end{array}\right.
 \end{align*}
 and $\Delta^{n/2}v_2=0$. It is easy to see 
$Ke^{nu}\in L^{\infty}(B_4(x_0))$ due to $K\in C^{k,\alpha}_{loc}$ and $u\in L^{\infty}_{loc}$. Hence, for any $p>0$,  
one has  $v_1\in W^{n,p}(B_4(x_0))$. Then for sufficienty large $p$, Sobolev embedding theorem yields that $v_1\in C^{n-1,\beta}(B_4(x_0))$ for some $0<\beta<1$, while $v_2$ is smooth in $B_4(x_0)$. Thus $u\in C^{n-1,\beta}(B_2(x_0))$. Then with the same procedure, we can bootsrap and make use of Schauder's estimate to prove that $u\in C^{n+k,\alpha}(B_1(x_0))$ see Theorem 9.19 in \cite{GT}.
 
\end{proof}


\begin{thebibliography}{99}
	\bibitem{ARS}
	Adimurthi, F.  Robert  and  M. Struwe;  \textit{Concentration phenomena for Liouville’s equation in dimension four.} J. Eur. Math. Soc. 8 (2006)171–180.
	\bibitem{Av}
	 P. Aviles;  \textit{Conformal complete metrics with prescribed nonnegative Gaussian curvature in $\mr^2$}.  Invent. Math. 83 (1986), no. 3, 519–544. 
		\bibitem{BGS}
	F. Borer,  L. Galimberti and M.  Struwe; \textit{ "Large'' conformal metrics of prescribed Gauss curvature on surfaces of higher genus}. Comment. Math. Helv. 90 (2015), no. 2, 407–428.
	 \bibitem{B-M}
	H. Brezis, F. Merle; \textit{Uniform estimates and blow-up behavior for solutions of $-\Delta u=V(x)e^u$ in two dimensions}, Commun. Partial Differ. Equ. 16(1991), 1223–1253.
	\bibitem{CY}
	A. S.-Y. Chang and P. C. Yang; \textit{Conformal deformation of metrics on $S^2$.} J. Differential Geom. 27 (1988), no. 2, 259–296.
	 \bibitem{CL91}
	W. Chen and C. Li; \textit{Classification of solutions of some nonlinear elliptic equations.} Duke Math. J. 63,615-622(1991).
	\bibitem{CL93}
	W. Chen and C.  Li; \textit{Qualitative properties of solutions to some nonlinear elliptic equations in $\textbf{R}^2$.} Duke Math. J. 71 (1993), no. 2, 427–439.
	\bibitem{CLO}
	W. Chen, C. Li and B. Ou, Biao; \textit{Classification of solutions for an integral equation. } Comm. Pure Appl. Math. 59 (2006), no. 3, 330–343. 
	\bibitem{CLin}
	K.-S. Cheng and C.-S Lin; \textit{On the asymptotic behavior of solutions of the conformal Gaussian curvature equations in $\textbf{R}^2$.} Math. Ann. 308 (1997), no. 1, 119–139. 
	\bibitem{Cheng-Ni}
	K.-S. Cheng and W.-M. Ni;  \textit{On the structure of the conformal Gaussian curvature equation on $\textbf{R}^2$}. Duke Math. J. 62 (1991), no. 3, 721–737. 
	\bibitem{Ga}
	L. Galimberti; \textit{ Compactness issues and bubbling phenomena for the prescribed Gaussian curvature equation on the torus}. Calc. Var. Partial Differential Equations 54 (2015), no.3,2483–2501.
	\bibitem{Ga17}
	L. Galimberti;  \textit{"Large'' conformal metrics of prescribed Q-curvature in the negative case.} NoDEA Nonlinear Differential Equations Appl. 24 (2017), no. 2, Paper No. 18, 36 pp.
	\bibitem{GT}
	D. Gilbarg and N.  Trudinger; \textit{Elliptic partial differential equations of second order. }Reprint of the 1998 edition. Classics in Mathematics. Springer-Verlag, Berlin, 2001. xiv+517 pp. 
	\bibitem{HM}
	 A. Hyder and L. Martinazzi;  \textit{Normal conformal metrics on $\mr^4$ with Q-curvature having power-like growth.} J. Differential Equations 301 (2021), 37–72.
	
	 \bibitem{L}
	 Y. Li; \textit{Harnack type inequality: the method of moving planes.} Comm. Math. Phys. 200 (1999), no. 2, 421–444.
	 \bibitem{Li04}
	 Y. Li; \textit{Remark on some conformally invariant integral equations: the method of moving spheres. } J. Eur. Math. Soc. (JEMS) 6 (2004), no. 2, 153–180. 
	 \bibitem{LS}
	 Y. Li and I. Shafrir;  \textit{Blow-up analysis for solutions of $-\Delta u=Ve^u$ in dimension two.} Indiana Univ. Math. J. 43 (1994), no. 4, 1255–1270.
	  \bibitem{LL}
	 E. Lieb and  M. Loss; \textit{Analysis}. Second edition. Graduate Studies in Mathematics, 14. American Mathematical Society, Providence, RI, 2001. xxii+346 pp. 
	 \bibitem{Lin}
	 C.-S Lin; \textit{A classification of solutions of a conformally invariant fourth order equation in $\textbf{R}^n$}. Comment. Math. Helv. 73 (1998), no. 2, 206–231.
	 \bibitem{Ma}
	 L. Martinazzi; \textit{Classification of solutions to the higher order Liouville's equation on $\mr^{2m}$}. Math. Z. 263 (2009), no. 2, 307–329. 
	 \bibitem{Mc}
	R.  McOwen; \textit{Conformal metrics in $R^2$ with prescribed Gaussian curvature and positive total curvature.} Indiana Univ. Math. J. 34 (1985), no. 1, 97–104.
	 \bibitem{Ni82}
	 W.M. Ni; \textit{On the elliptic equation $\Delta u+K(x)e^{2u}=0$ and conformal metrics with prescribed Gaussian curvatures.} Invent. Math. 66 (1982), no. 2, 343–352
	 \bibitem{NZ}
	 Q.A. Ngô and  H. Zhang; \textit{Bubbling of the prescribed Q-curvature equation on 4-manifolds in the null case} arXiv :1903 .12054, 2019.
	 \bibitem{Pi}
	 P. Pizzetti,; \textit{Sulla media dei valori che una funzione dei punti dello spazio assume alla superficie di una sfera.} Rend. Lincei 18, 182–185 (1909).
	 \bibitem{Struwe05}
	 M. Struwe;  \textit{A flow approach to Nirenberg's problem.} Duke Math. J. 128 (2005), no. 1, 19–64. 
\bibitem{Struwe20}
M. Struwe; \textit{"Bubbling" of the prescribed curvature flow on the torus}, J. Eur. Math. Soc. 22(2020), 3223–3262.
\bibitem{Struwe21}
M. Struwe;  \textit{A Liouville-type result for a fourth order equation in conformal geometry.} Vietnam J. Math. 49 (2021), no. 2, 267–279.
\bibitem{WX}
J. Wei and X. Xu;  \textit{Classification of solutions of higher order conformally invariant equations.} Math. Ann. 313 (1999), no. 2, 207–228. 
\bibitem{Xu05}
X. Xu; \textit{Uniqueness and non-existence theorems for conformally invariant equations}. J. Funct. Anal. 222 (2005), no.1, 1–28. 
\bibitem{XY}
X. Xu and P. C. Yang; \textit{Remarks on prescribing Gauss curvature.} Trans. Amer. Math. Soc. 336 (1993), no. 2, 831–840. 
\end{thebibliography}
\end{document}